\documentclass[11pt]{amsart}

\usepackage{amssymb}

\input xy

\xyoption{all}

\theoremstyle{plain}

\newtheorem{theorem}{Theorem}[section]

\newtheorem{lemma}[theorem]{Lemma}

\newtheorem{proposition}[theorem]{Proposition}

\newtheorem{fact}[theorem]{Fact}

\theoremstyle{definition} 

\newtheorem{definition}[theorem]{Definition}

\newtheorem{example}[theorem]{Example}

\newtheorem{remark}[theorem]{Remark}

\newtheorem{hypothesis}[theorem]{Hypothesis}

\theoremstyle{remark}

\renewcommand{\phi}{\varphi}

\newcommand{\initial}\lessdot

\newcommand{\id}{\operatorname{id}}

\def\?{?\vadjust

{\vbox to 0pt{\vskip-7pt\hbox to 1.1\hsize{\hfill\huge ?!}}}}

\usepackage{verbatim}

 \def\nfork{\setbox0\hbox{$\bigcup$}%
 \setbox1=\hbox to \wd0{\hfil\vrule width 0.7pt depth 2pt height 7.5pt\hfil}%
 \wd1=0cm\relax\box1\box0}
 \def\dnf{\mathop{\nfork}\limits}

\newcommand{\be}{\begin{enumerate}}
\newcommand{\ee}{\end{enumerate}}
\newcommand{\case}{\emph}

\newcommand{\discussion}{\subsubsection*}
\renewcommand{\epsilon}{\varepsilon}

\newcommand{\al}{\allowbreak}
\usepackage{enumerate}

\begin{document}
\title{Independence of Sets Without Stability}

\thanks{This paper is a part of the PHD thesis of the first author
with the supervisor of Boaz Tsaban}

\author{Adi Jarden}
\email[Adi Jarden]{jardenadi@gmail.com}
\address{Department of Mathematics and computer science.\\ Bar-Ilan University \\ Ramatgan 52900, Israel}

\author{Alon Sitton}
\email[Alon Sitton]{alonsitton@gmail.com}
\address{Institute of Mathematics\\ Hebrew University\\ Jerusalem, Israel}

\maketitle

\begin{abstract}
We presents an independence relation on sets, one can define
dimension by it, assuming that we have an abstract elementary
class with a forking notion that satisfies the axioms of a good
frame minus stability.
\end{abstract}

\tableofcontents

\discussion{Introduction} We would like to find an ``independence
relation'' such that if $K$ is a class of models, $M_0,M_1 \in K$
and $J \subseteq M_1-M_0$ then we can say if $J$ is independent in
$M_1$ over $M_0$ or not.
The independence relation should satisfy the following things:
\begin{enumerate}
\item If $K$ is the class of fields with character
0, then independence is linear independence.
\item If $J^* \subseteq J \subseteq M_1-M_0$ and $J$ is
independent in $M_1$ over $M_0$ then $J^*$ is independent in $M_1$
over $M_0$. \item If $J_1,J_2$ are maximal independent subsets of
$M_1$ over $M_0$ then $|J_1|=|J_2|$ or they both finite.
\end{enumerate}
It is hard to find such a relation, so we decrease our ambitions
in two aspects:
\begin{enumerate}
\item $K$ should be a ``good class'' of models. \item $M_0$ should
be a ``good submodel'' of $M_1$.
\end{enumerate}
The restriction of the context is made in two steps: In section 1,
we define a list of axioms, a pair $(K,\preceq_\frak{k})$ that
satisfy those axioms is called an ``abstract elementary class''
(in short a.e.c.). In section 2 we define a list of axioms of a
non-forking relation and restrict the study to those a.e.c.'s, one
can find a non-forking relation related to them. On such a.e.c.'s
we exhibit in section 3 an independence relation and prove that
one can define dimension by it. Sections 1,2,3 are self contained.
In section 4 assuming existence of uniqueness triples (so
familiarity with \cite{jrsh875}, \cite{sh600} or \cite{sh705} is
assumed), We prove that the relations independence and finitely
independence are the same and more properties.

What are the connections between the present paper and other
papers? In \cite{jrsh875} we study stability theory without
assuming stability, but weak stability. The main purpose is to
study abstract elementary classes (shortly a.e.c.'s) which are
$PC_{\aleph_0}$. The theorems we prove here, may be useful in the
study of such classes too.



The frame we define (``good frame minus stability'') is similar to
the weak forking notion which is defined in \cite{grko}, which is
parallel to simple first order theories, see Remark 1.5 on page 8
of \cite{grko} (but there is a significant difference: here we
work in one cardinality only).

We define independence as in \cite{sh705}. In what aspects do we
improve here the results in section 5 of \cite{sh705}?
\begin{enumerate}
\item We do not assume
stability. \item We do not
assume successfulness. \item We do not assume $goodness^+$. \item
We prove several important propositions without assuming that
$K^{3,uq}$ has existence.
\end{enumerate}

\section{Abstract Elementary Classes}
\begin{definition}[Abstract Elementary Classes]\label{1.1}
\mbox{}
\begin{enumerate}
\item  Let $K$ be a class of models for a fixed vocabulary and let
$\preceq=\preceq_\frak{k}$ be a 2-place relation on $K$. The pair
$\frak{k}=(K,\preceq_\frak{k})$ is an \emph{a.e.c.} if the
following axioms are satisfied:
\begin{enumerate}
\item  $K,\preceq$ are closed under isomorphisms. In other words,
if $M_1 \in K$, $M_0 \preceq_{\frak{k}}M_1$ and $f:M_1 \to N_1$ is
an isomorphism then $N_1 \in K$ and $f[M_0]\preceq_\frak{k}N_1$.
\item $\preceq$ is a partial order and it is included in the
inclusion relation. \item If $\langle
M_\alpha:\alpha<\delta\rangle$ is a continuous
$\preceq_\frak{k}$-increasing sequence, then $$M_0 \preceq \bigcup
\{M_\alpha:\alpha<\delta\} \in K.$$ \item Smoothness: If $\langle
M_\alpha:\alpha<\delta \rangle$ is a continuous
$\preceq_\frak{k}$-increasing sequence, and for every
$\alpha<\delta,\ M_\alpha\preceq N$, then
$$\bigcup \{M_\alpha:\alpha<\delta\} \preceq N.$$
\item If $M_0 \subseteq M_1 \subseteq M_2$ and $M_0 \preceq M_2
\wedge M_1 \preceq M_2$, then $M_0 \preceq M_1$. \item  There is a
Lowenheim Skolem Tarski number, $LST(\frak{k}$), which is the
minimal cardinal $\lambda$, such that for every model $N \in K$
and a subset $A$ of it, there is a model $M \in K$ such that $A
\subseteq M \preceq N$ and the cardinality of $M$ is $\leq
\lambda+|A|$.
\end{enumerate}
\item  $\frak{k}=(K,\preceq)$ is an \emph{a.e.c. in $\lambda$} if:
The cardinality of every model in $K$ is $\lambda$, and it
satisfies axioms a,b,d,e of a.e.c., and for sequences $\langle
M_\alpha:\alpha<\delta \rangle$ with $\delta<\lambda^+$ it
satisfies axiom c too.
\end{enumerate}
\end{definition}

In \cite{gr21} there are examples of a.e.cs. The following are
examples of naturals classes which are not a.e.cs.

\begin{example}
The class of \emph{sets} (i.e. models without relations or
functions) with cardinality less than $\kappa$, where $\aleph_0
\leq \kappa$ and the relation is $\subseteq$, is \emph{not} an
a.e.c., as it does not satisfy axiom c.

The class of sets with the relation $\preceq=\{(M,N):M \subseteq
N$ and $||N-M||>\kappa\}$ where $\aleph_0 \leq \kappa$, is not an
a.e.c., as it does \emph{not} satisfy smoothness (axiom d).
\end{example}

\begin{definition}\label {1.2}
\mbox{} 
\begin{enumerate}
\item For a model $M \in K$ we denote its universe by $|M|$, and
its cardinality by $||M||$. \item $K_\lambda=:\{M \in
K:||M||=\lambda\}$.
\end{enumerate}
\end{definition}

\begin{definition}
\label{1.2b} \mbox{}
\begin{enumerate}
\item Let $M,N$ be models in $K$ and let $f$ be an injection of
$M$ to $N$. We say that $f$ is a
$\preceq_\frak{k}$-\emph{embedding}, or $f$ is an embedding (if
$\preceq_\frak{k}$ is clear from the context), when $f$ is an
injection with domain $M$ and $Im(f)\preceq_\frak{k}N$. \item A
function $f:B \to C$ is \emph{over} $A$, if $A \subseteq B \bigcap
C$ and $x \in A \Rightarrow f(x)=x$.
\end{enumerate}
\end{definition}

\begin{definition}
\label{1.3} \mbox{}
\begin{enumerate}
\item $K^3=:\{(M,N,a):M \in K,\ N \in K,\ M \preceq N,\ a \in
N\}$. \item $K^3_\lambda=:\{(M,N,a):M \in K_\lambda,\ N \in
K_\lambda,\ M \preceq N,\ a \in N\}$. \item $E^*=E^*_k$ is the
following relation on $K^3$: $(M_0,N_0,a_0)E^*(M_1,N_1,a_1)$ iff
$M_1=M_0$ and for some $N_2,f$ we have: $N_1 \preceq N_2,\ f:N_0
\to N_2$ is an embedding over $M_0$ and $f(a_0)=a_1$. \item
$E^*_\lambda:=E^* \restriction K^3_\lambda$. \item $E=E_k$ is the
closure of $E^*$ under transitivity, i.e. the closure to an
equivalence relation.
\end{enumerate}
\end{definition}

\begin{definition}\label{1.4}\label{amalgamation}
\mbox{}
\begin{enumerate} \item We say that $\frak{k}_\lambda$
has \emph{amalgamation} when: For every $M_0,M_1,M_2$ in
$K_\lambda$, such that $n<3 \Rightarrow M_0 \preceq_{\frak{k}}
M_n$, for some $(f_1,f_2,M_3)$ we have: $f_n:M_n \to M_3$ is an
embedding over $M_0$, i.e. the diagram below commutes. In such a
case we say that that $(f_1,f_2,M_3)$ is an amalgamation of
$M_1,M_2$ over $M_0$ or that $M_3$ is an amalgam of $M_1,M_2$ over
$M_0$.

\[ \xymatrix{\ar @{} [dr]
M_1
\ar[r]^{f_1}  &M_3 \\
M_0 \ar[u]^{\id} \ar[r]_{\id} & M_2 \ar[u]_{f_2} }
\]

\item we say that $K_\lambda$ has \emph{joint embedding} when: If
$M_1,M_2 \in K_\lambda$, then there are $f_1,f_2,M_3$ such that
for $n=1,2$ $f_n:M_n \to M_3$ is an embedding and $M_3 \in
K_\lambda$.
\item $M \in K$ is $\preceq_\frak{k}$-\emph{maximal} if there is
no $N \in K$ such that $M\prec N$.
\end{enumerate}
\end{definition}

\begin{proposition} \label{1.5}
\mbox{}
\begin{enumerate} \item $(M,N_0,a)E^*(M,N_1,b)$ iff there
is an amalgamation $N,f_0,f_1$ of $N_0$ and $N_1$ over $M$ such
that $f_0(a)=f_1(b)$. \item $E^*$ is a reflexive, symmetric
relation. \item If $\frak{k}$ has amalgamation, then $E^*$ is an
equivalence relation. \item If $\frak{k}_\lambda$ has
amalgamation, then $E^*_\lambda$ is an equivalence relation.
\end{enumerate}
\end{proposition}
\begin{proof}
Easy.
\end{proof}

\begin{definition} \label{1.6}
\mbox{}
\begin{enumerate}
\item For $(M,N,a) \in K^3$ let $tp(a,M,N)=tp_k(a,M,N)$, the
\emph{type} of $a$ in $N$ over $M$, be the equivalence class of
$(M,N,a)$ under $E$ (In other texts, it is called
``$ga-tp(a/M,N)$''). \item $S(M)=S_k(M)=\{tp(a,M,N):(M,N,a) \in
K^3\}$. \item If $M_0 \preceq M_1, p \in S(M_1)$ then define
$p\restriction M_0=tp(a,M_0,N)$, (by the definitions of $E,E^*$ it
is easy to check that $p\restriction M_0$ does not depend on the
representative of $p$). \item If $p=tp(a,M,N)$ and $f:M \to M^*$
is an isomorphism, then we define $f(p):=tp(f(a),f[M],f^+[N])$
where $f^+$ is an extension of $f$ ($f(p)$ does not depend on the
choice of $f^+$).
\end{enumerate}
\end{definition}

\section{Good Frame Minus Stability}
We define the non-forking frame, we are going to work with. It is
similar to the good frame that Shelah defined in section 2 of
\cite{sh600}, but here we do not assume basic stability.
\begin{definition} \label{definition of a good frame minus stability}
$\frak{s}=(\frak{k},S^{bs},\dnf)$ is a \emph{good $\lambda$-frame minus stability} if:\\
(1) $\frak{k}=(K,\preceq_\frak{k})$ is an a.e.c., $LST(\frak{k})
\leq \lambda$, and the following four axioms are satisfied in
$\frak{k}_\lambda$: It has joint embedding, amalgamation and there
is no $\preceq$-maximal model in $\frak{k}_\lambda$.\\
(2) $S^{bs}$ is a function with domain $K_\lambda$, which
satisfies the following axioms:
\begin{enumerate}[(a)] \item It respects isomorphisms. \item $S^{bs}(M) \subseteq
S^{na}(M)=:\{tp(a,M,N):M\prec N \in K_\lambda,\ a\in N-M\}$. \item
Densite of basic types: If $M\prec N$ in $K_\lambda$, then there
is $a \in N-M$ such that $tp(a,M,N) \in S^{bs}(M)$.
\end{enumerate}
(3) the relation $\dnf$ satisfies the following axioms:
\begin{enumerate}[(a)]
\item $\dnf$ is a subset of $\{(M_0,M_1,a,M_3):n\in \{0,1,3\}
\Rightarrow M_n \in K_\lambda,\ a \in M_3-M_1,\ n<2 \Rightarrow
tp(a,M_n,M_3) \in S^{bs}(M_n)\}$. \item Monotonicity: If $M_0
\preceq M^*_0 \preceq M^*_1 \preceq M_1 \preceq M_3,\ M^*_1\bigcup
\{a\} \subseteq M^{**}_3 \preceq M^*_3$, then $\dnf(M_0,M_1,a,M_3)
\Rightarrow \dnf(M^*_0,M^*_1,a,M^{**}_3)$. [So we can say ``$p$
does not fork over $M_0$'' instead of $\dnf(M_0,M_1,a,M_3)$].
\item Local character: If $\langle M_\alpha:\alpha \leq \delta
\rangle$ is an increasing continuous sequence, and
$tp(a,M_\delta,M_{\delta+1}) \in S^{bs}(M_\delta)$, then there is
$\alpha<\delta$ such that $tp(a,M_\delta,\allowbreak
M_{\delta+1})$ does not fork over $M_\alpha$. \item Uniqueness of
the non-forking extension: If $p,q \in S^{bs}(N)$ do not fork over
$M$, and $p\restriction M=q\restriction M$, then p=q. \item
Symmetry: If $M_0 \preceq M_1 \preceq M_3,\ a_1 \in M_1,\
tp(a_1,M_0,M_3) \in S^{bs}(M_0)$, and $tp(a_2,M_1,M_3)$ does not
fork over $M_0$, \emph{then} for some $M_2,M^*_3$, $a_2 \in M_2,\
M_0 \preceq M_2 \preceq M^*_3,\ M_3 \preceq M^*_3$, and
$tp(a_1,M_2,M^*_3)$ does not fork over $M_0$. \item Existence of
non-forking extension: If $p \in S^{bs}(M)$ and $M\prec N$, then
there is a type $q \in S^{bs}(N)$ such that $q$ does not fork over
$M$ and $q\restriction M=p$. \item Continuity: Let $\langle
M_\alpha:\alpha \leq \delta \rangle$ be an increasing continuous
sequence. Let $p \in S(M_\delta)$. If for every $\alpha \in
\delta,\ p\restriction M_\alpha$ does not fork over $M_0$, then $p
\in S^{bs}(M_\delta)$ and does not fork over $M_0$.
\end{enumerate}
\end{definition}

\begin{proposition}[versions of axiom f]
\label{versions of extension}  If for $n<3$ $M_n \in K_\lambda$,
$M_0 \preceq M_n$, and $tp(a,M_1,M_0) \in S^{bs}(M_0)$ then:
\begin{enumerate}
\item There are $M_3,f$ such that:
\begin{enumerate}
\item $M_2 \preceq M_3$. \item $f:M_1 \to M_3$ is an embedding
over $M_0$. \item $tp(f(a),M_2,M_3)$ does not fork over $M_0$.
\end{enumerate}
\item There are $M_3,f$ such that:
\begin{enumerate}
\item $M_1 \preceq M_3$. \item $f:M_2 \to M_3$ is an embedding
over $M_0$. \item $tp(a,f[M_2],M_3)$ does not fork over $M_0$.
\end{enumerate}
\end{enumerate}
\end{proposition}
\begin{proof}
Easy.
\end{proof}

\begin{proposition}[transitivity] \label{transitivity}
Suppose $\frak{s}$ is a semi-good $\lambda$-frame minus stability.
If $M_0 \preceq M_1 \preceq M_2,\ p \in S^{bs}(M_2)$ does not fork
over $M_1,\ p\restriction M_1$ does not fork over $M_0$, then $p$
does not fork over $M_0$.
\end{proposition}
\begin{proof} Suppose $M_0\prec M_1\prec M_2,\ n<3 \Rightarrow M_n \in
K_\lambda,\ p_2 \in S^{bs}(M_2)$ does not fork over $M_1$ and
$p_2\restriction M_1$ does not fork over $M_0$. For $n<2$ define
$p_n=p_2\restriction M_n$. By Definition \ref{definition of a good
frame minus stability}.f there is a type $q_2 \in S^{bs}(M_2)$
such that $q_2\restriction M_0=p_0$ and $q_2$ does not fork over
$M_0$. Define $q_1=q_2\restriction M_1$. By Definition
\ref{definition of a good frame minus stability}.b (monotonicity)
$q_1$ does not fork over $M_0$. So by Definition \ref{definition
of a good frame minus stability}.d (uniqueness) $q_1=p_1$. Using
again Definition \ref{definition of a good frame minus
stability}.f, we get $q_2=p_2$, as they do not fork over $M_1$. By
the definition of $q_2$ it does not fork over $M_0$.
\end{proof}

\begin{fact} \label{axiom i is abandonment}
Suppose
\begin{enumerate}
\item $\frak{s}$ is a semi-good $\lambda$-frame minus stability.
\item $n<3 \Rightarrow M_0 \preceq M_n$. \item For $n=1,2$, $a_n
\in M_n-M_0\ and \ tp(a_n,M_0,M_n) \allowbreak \in S^{bs}(M_0)$.
\end{enumerate}
Then there is an amalgamation $M_3,f_1,f_2$ of $M_1,M_2$ over
$M_0$ such that for $n=1,2$ $tp(f_n(a_n),f_{3-n}[M_{3-n}],M_3)$
does not fork over $M_0$.
\end{fact}

For completeness we give a proof:
\begin{proof} Suppose for $n=1,2$ $M_0\prec M_n \wedge tp(a_n,M_0,M_n)
\in S^{bs}(M_0)$. By Proposition \ref {versions of extension}.1,
there are $N_1,f_1$ such that:
\begin{enumerate} \item $M_1 \preceq N_1$. \item $f_1:M_2
\to N_1$ is an embedding over $M_0$. \item $tp(f_1(a_2),M_1,N_1)$
does not fork over $M_0$.
\end{enumerate}

By Definition \ref{definition of a good frame minus stability}.f
(the symmetry axiom), there are a model $N_2$, $N_1 \preceq N_2
\in K_\lambda$ and a model $N_2^* \in K_\lambda$ such that: $M_0
\bigcup \{f_1(a_2)\} \subseteq N_2^* \preceq N_2$ and
$tp(a_1,N_2^*,N_2)$ does not fork over $M_0$.\\
By Claim \ref{versions of extension}.2 (substituting
$N_2^*,N_2,N_2,a_1$ which appear here instead of $M_0,M_1,M_2,a$
there) there are $N_3,f_2$ such that:
\begin{enumerate}
\item $N_2 \preceq N_3$. \item $f_2:N_2 \to N_3$ is an embedding
over $N_2^*$. \item $tp(a_1,f_2[N_2],N_3)$ does not fork over
$N_2^*$.
\end{enumerate}
So by Claim \ref{transitivity} (page \pageref{transitivity}),
$tp(a_1,f_2[N_2],N_3)$ does not fork over $M_0$. So as $M_0
\preceq f_2 \circ f_1[M_2] \preceq f_2[N_2]$ by Definition
\ref{definition of a good frame minus stability}.b (monotonicity)
$tp(a_1,f_2 \circ f_1[M_2],N_3)$ does not fork over $M_0$. As $f_2
\restriction N_2^*=id_{N_2^*}$, $f_2(f_1(a_1))=f_1(a_1)$.
\end{proof}

\begin{fact}[\cite{sh600} section 4] \label{3.3}
\label{amalgamation of a model and a sequence}
Let $\langle M_\alpha:\alpha \leq \theta \rangle$ be an increasing
continuous sequence of models in $\frak{k}_\lambda$. Let $N \succ
M_0$, and for $\alpha<\theta$, let $a_\alpha \in
M_{\alpha+1}-M_\alpha,\ (M_\alpha,M_{\alpha+1},a_\alpha) \in
K^{3,bs}$ and $b \in N-M_0,\ (M_0,N,b) \in K^{3,bs}$. Then for
some $f,\langle N_\alpha:\alpha \leq \theta \rangle$ the following
hold:

\begin{displaymath}
\xymatrix{
N \ar[r]^{f} & N_1 \ar[r] & N_2 \ar[rr] && N_\alpha \ar[r] & N_{\alpha+1} \ar[rr] && N_\theta\\
M_0 \ar[r] \ar[u] &  M_1  \ar[r] \ar[u] & M_2  \ar[rr] \ar[u] &&
M_\alpha  \ar[r] \ar[u] & M_{\alpha+1}  \ar[rr] \ar[u] &&
M_\theta}
\end{displaymath}

\begin{enumerate}[(a)]
\item $f$ is an isomorphism of $N$ to $N_1$ over $M_0$. \item
$\langle N_\alpha:\alpha \leq \theta \rangle$ is an increasing
continuous sequence. \item $M_\alpha \preceq N_\alpha$. \item
$tp(a_\alpha,N_\alpha,N_{\alpha+1})$ does not fork over $M_0$.
\item $tp(f(b),M_\alpha,N_\alpha)$ does not fork over $M_0$.
\end{enumerate}
\end{fact}

For completeness, we give a proof:
\begin{proof}
First we explain the idea of the proof. Suppose $M_0 \allowbreak
\preceq M_1,\ M_0 \preceq M_2$. Then there is an amalgamation
$M_3,f_1,f_2$ of $M_1,M_2$ over $M_0$. such that $f_1=id_{M_1}$.
There is also such an amalgamation such that $f_2=id_{M_2}$. But
maybe there is no such an amalgamation such that $f_1=id_{M_1}$
and $f_2=id_{M_2}$. So we have to choose if we want to "fix" $M_1$
or $M_2$. In our case we have to amalgamate $N$ with another model
$\theta$ times. So if we want to "fix" the models in the sequence
$\langle M_\alpha:\alpha \leq \theta \rangle$, then we will
"change" $N$ $\theta$ times. So in limit steps we will be in a
problem. The solution is to fix $N$, and "change" the sequence
$\langle M_\alpha:\alpha \leq \theta \rangle$. In the end of the
proof we "return the sequence to its place".\\
\emph{The proof itself:} We choose $(N_\alpha^*,f_\alpha)$ by
induction on $\alpha$ such that:
\begin{enumerate}
\item $\alpha \leq \theta \Rightarrow N_\alpha^* \in K_\lambda$.
\item $(N_0^*,f_0)=(N,id_{M_0})$. \item The sequence $\langle
N_\alpha^*:\alpha \leq \theta \rangle$ is increasing and
continuous. \item The sequence $\langle f_\alpha:\alpha \leq
\theta \rangle$ is increasing and continuous. \item For $\alpha
\leq \theta$, the function $f_\alpha$ is an embedding of
$M_\alpha$ to $N_\alpha^*$. \item
$tp(f_\alpha(a_\alpha),N_\alpha^*,N_{\alpha+1}^*)$ does not fork
over $f_\alpha[M_\alpha]$. \item
$tp(b,f_\alpha[M_\alpha],N_\alpha^*)$ does not fork over $M_0$.
\end{enumerate}
Why is this possible? For $\alpha=0$ see 2. For $\alpha$ limit
define $N_\alpha^*:=\bigcup \{N_\beta^*:\beta<\alpha \},\
f_\alpha:=\bigcup \{f_\beta:\beta<\alpha \}$. By the induction
hypothesis $\beta<\alpha \Rightarrow f_\beta[M_\beta] \preceq
N_\beta^*$ and the sequences $\langle N_\beta^*:\beta \preceq
\alpha \rangle,\ \langle f_\beta:\beta \preceq \alpha \rangle$ are
increasing and continuous. So by the smoothness (Definition
\ref{1.1}.d) $f_\alpha[M_\alpha] \preceq N_\alpha^*$. By the
induction hypothesis for $\beta \in \alpha$ the type
$tp(b,f_\beta[M_\beta],N_\beta^*)$ does not fork over $M_0$. So by
Definition \ref{definition of a good frame minus stability}.h
(continuity), the type $tp(b,f_\alpha[M_\alpha],N_\alpha^*)$ does
not fork over $M_0$. Why can we define $(N_\alpha^*,f_\alpha)$ for
$\alpha=\beta+1$? Let $f_{\beta+0.5}$ be a function with domain
$M_\alpha$ which extend $f_\beta$. By condition 5 in the induction
hypothesis,
 $f_\beta[M_\beta]
 \preceq N_\beta^*$, and obviously $f_\beta[M_\beta] \preceq
f_{\beta+0.5}[M_\alpha]$. By assumption,
$tp(a_\beta,M_\beta,M_\alpha) \in S^{bs}(M_\beta)$. So
$tp(f_{\beta+0.5}(a_\beta),f_\beta[M_\beta],f_{\beta+0.5}[M_\alpha])
\in S^{bs}(f_\beta[M_\beta])$. By condition 7 in the induction
hypothesis, $tp(b,f_\beta[M_\beta],N_\beta^*) \in
S^{bs}(f_\beta[M_\beta])$. So by Definition \ref{definition of a
good frame minus stability}.f, there are a model $N_{\beta+1}
\succ N_\beta$ and an embedding $f_{\beta+1} \supseteq f_{\beta}$
such that condition 6 is satisfied and the type
$tp(b,f_{\beta+1}[M_{\beta+1}],N_{\beta+1}^*)$ does not fork over
$f_\beta[M_\beta]$. By Proposition \ref{transitivity} condition 7
is satisfied. So we can choose by induction $N_\alpha^*,f_\alpha$.

Now $f_\theta:M_\theta \to N_\theta^*$ is an embedding. Extend
$f_\theta^{-1}$ to a function g with domain $N_\theta^*$ and
define $f:=g \restriction N$. By 2,3 $N \preceq N_\theta^*$. By 2,
$f$ is an isomorphism over $M_0$, so requirement a is satisfied.
Define $N_\alpha:=g[N_\alpha^*]$. By 5, $f_\alpha[M_\alpha]
\preceq N_\alpha^*$, so $M_\alpha \preceq N_\alpha$. So
requirement c is satisfied. It is easy to see that 3 implies
requirement b and 6,7 implies requirements e,f.
\end{proof}

\section{Independence} From now until the end of the paper we
assume that:
\begin{hypothesis}
 $\frak{s}$ is a good $\lambda-$frame minus stability.
\end{hypothesis}

\begin{definition} \label{1.6} \label{definition of independence}
\mbox{}

(a) $\langle M_\alpha,a_\alpha:\alpha<\alpha^* \rangle ^\frown
\langle M_{\alpha^*} \rangle$ is said to be \emph{independent}
over $M$ when:
\begin{enumerate}
\item $\langle M_\alpha:\alpha \leq \alpha^* \rangle$ is an
increasing continuous sequence of models in $\frak{k}_\lambda$.
\item $M \preceq M_0$. \item For $\alpha<\alpha^*$, the type
$tp(a_\alpha,M_\alpha,M_{\alpha+1})$ does not fork over $M$.
\end{enumerate}

(b) $\langle a_\alpha:\alpha<\alpha^* \rangle$ is said to be
\emph{independent} in $(M,M_0,M^+)$ when for some increasing
continuous sequence $\langle M_\alpha:0<\alpha \leq \alpha^*
\rangle$ and a model $M^{++}$ the sequence $\langle
M_\alpha,a_\alpha:\alpha<\alpha^* \rangle ^\frown \langle
M_{\alpha^*} \rangle$ is independent over $M$ and $M^+ \succeq
M^{++}
 \preceq M_{\alpha^*}$.

(c) The \emph{set} $J$ is said to be \emph{independent} in $(M,A,N)$ when:\\
There is an independent sequence $\langle
M_\alpha,a_\alpha:\alpha<\alpha^* \rangle ^\frown \langle
M_{\alpha^*} \rangle$ over $M$ such that:
\begin{enumerate}
\item $J=\{a_\alpha:\alpha<\alpha^*\}$. \item $A \subset M_0$.
\item There is a model $N^+$ such that $M_{\alpha^*} \preceq N^+$
and $N \preceq N^+$.
\end{enumerate}

(d) $J$ is said to be \emph{finitely independent} in $(M,A,N)$
when: every finite subset of $J$ is independent in $(M,A,N)$.
\end{definition}

\begin{definition}\label{definition of finite character}
We say that \emph{independence is finitely independence} when for
every $M,M_0,M^+,J$ the following hold: $J$ is independent in
$(M,M_0,M^+)$ iff $J$ is finitely independent in $(M,M_0,M^+)$.
\end{definition}

\begin{proposition}\label{amalgamation of a model and an
independent sequence} If the sequence $\langle
a_\alpha:\alpha<\alpha^* \rangle$ is independent in $(M,M_0,M_2)$
and $tp(a,M_0,M_1) \in S^{bs}(M_0)$ then for some amalgamation
$(id_{M_2},f,M_3)$ of $M_2,M_1$ over $M_0$ the sequence $\langle
a_\alpha:\alpha<\alpha^* \rangle$ is independent in
$(M,f[M_1],M_3)$ and $tp(f(a),M_2,M_3)$ does not fork over $M_0$.
\end{proposition}

\begin{proof}
By Fact \ref{3.3}
\end{proof}

\begin{proposition}\label{adding a type to an independent
sequence} If $\langle M_\alpha,a_\alpha:\alpha<\alpha^* \rangle
^\frown \langle M_{\alpha^*} \rangle$ is an independent sequence
over $M$ and $p \in S^{bs}(M_{\alpha^*})$, then there is a
sequence $\langle N_\alpha:\alpha \leq \alpha^* \rangle$ such
that:
\begin{enumerate}
\item For $\alpha \leq \alpha^*$, $M_\alpha \preceq N_\alpha$.
\item For $\alpha<\alpha^*$ the type
$tp(a_\alpha,N_\alpha,N_{\alpha+1})$ does not fork over
$M_\alpha$. \item The sequence $\langle
N_\alpha,a_\alpha:\alpha<\alpha^* \rangle ^\frown \langle
N_{\alpha^*} \rangle$ is independent over $M$. \item $p$ is
realized in $N_{\alpha^*}$.
\end{enumerate}
\end{proposition}

\begin{proof}
First note that clause 3 follows by clauses 1,2 and the
transitivity proposition (Proposition \ref{transitivity}). By the
Definition \ref{definition of a good frame minus stability}.c
(local character axiom), there is $\alpha_0<\alpha^*$ such that
$p$ does not fork over $M_{\alpha_0}$. So there is $N \succ
M_{\alpha_0}$ which realizes $p \restriction M_{\alpha_0}$, namely
there is an element $a \in N$ such that $tp(a,M_{\alpha_0},N)=p
\restriction M_{\alpha_0}$. By Proposition \ref{amalgamation of a
model and a sequence} there a sequence $\langle N_\alpha:\alpha
\in [\alpha_0,\alpha^*] \rangle$ such that clauses 1,2 hold,
$tp(a,M_{\alpha^*},N_{\alpha^*})$ does not fork over
$M_{\alpha_0}$ and $N_{\alpha_0},N$ are isomorphic over
$M_{\alpha_0}$. By Definition \ref{definition of a good frame
minus stability}.d (uniqueness)
$tp(a,M_{\alpha^*},N_{\alpha^*})=p$. Now define
$N_\alpha:=N_{\alpha_0}$ for $\alpha<\alpha_0$.
\end{proof}

\begin{proposition}\label{preparation exchange}\label{increasing
the big model} If $\langle M_\alpha,a_\alpha:\alpha<\alpha^*
\rangle ^\frown \langle M_{\alpha^*} \rangle$ is independent over
$M$ and $M_{\alpha^*} \prec M^+$, then there is a sequence
$\langle N_\alpha:\alpha \leq \alpha^* \rangle$ such that:
\begin{enumerate}
\item For $\alpha \leq \alpha^*$, $M_\alpha \preceq N_\alpha$.
\item $N_0=M_0$. \item The sequence $\langle
N_\alpha,a_\alpha:\alpha<\alpha^* \rangle ^\frown \langle
N_{\alpha^*} \rangle$ is independent over $M$. \item $M^+ \preceq
N_{\alpha^*}$.
\end{enumerate}
\end{proposition}

\begin{displaymath}
\xymatrix{N_0 \ar[r]^{id} & N_1 \ar[r]^{id} & N_\alpha \ar[r]^{id} & N_{\alpha+1} \ar[r]^{id} & N_{\alpha^*}\\
&&&& M^+ \ar[u]^{id} \\
M_0 \ar[r]^{id} \ar[uu]^{id} & M_1 \ar[r]^{id} \ar[uu]^{id} &
M_\alpha \ar[r]^{id} \ar[uu]^{id} & M_{\alpha+1} \ar[r]^{id}
\ar[uu]^{id} & M_{\alpha^*} \ar[u]^{id}}
\end{displaymath}

\begin{proof}
The idea is to find $\lambda^+$ candidates for $\langle
N_\alpha:\alpha \leq \alpha^* \rangle$. If none of them satisfies
the conclusion, then we get a contradiction.

We try to choose by induction on $\beta<\lambda^+$ a candidate
$\langle M_{\beta,\alpha} : \alpha \allowbreak \leq \alpha^*+1
\rangle$ such that:
\begin{enumerate}[(i)]
\item $\langle M_{\beta,\alpha}:\alpha \leq \alpha^*+1 \rangle$ is
increasing and continuous. \item For $\alpha \leq \alpha^*$
$M_{0,\alpha}=M_\alpha$. \item $M_{0,\alpha^*+1}:=M^+$. \item For
$\alpha \leq \alpha^*+1$, $M_{\beta,\alpha} \preceq
M_{\beta+1,\alpha}$. \item For $\alpha<\alpha^*$,
$tp(a_\alpha,M_{\beta+1,\alpha},M_{\beta+1,\alpha+1})$ does not
fork over $M_{\beta,\alpha}$. \item If $\beta$ is limit then for
$\alpha \leq \alpha^*+1$ $M_{\beta,\alpha}=\bigcup
\{M_{\gamma,\alpha}:\gamma<\beta\}$. \item $M_{\beta+1,\alpha^*}
\bigcap M_{\beta,\alpha^*+1} \neq M_{\beta,\alpha^*}$.
\end{enumerate}

\begin{displaymath}
\xymatrix{M_{\beta+1,0} \ar[r]^{id} & M_{\beta+1,\alpha}
\ar[r]^{id} & M_{\beta+1,\alpha+1} \ar[r]^{id} & M_{\beta+1,\alpha^*}  \ar[r]^{id} & M_{\beta+1,\alpha^*+1}\\
M_{\beta,0} \ar[r]^{id} \ar[u]^{id} \ar[u]^{id} & M_{\beta,\alpha}
\ar[r]^{id} \ar[u]^{id} & M_{\beta,\alpha+1} \ar[r]^{id} \ar[u]^{id} & M_{\beta,\alpha^*}  \ar[r]^{id} \ar[u]^{id} & M_{\beta,\alpha^*+1} \ar[u]^{id}\\
 M_{0,0}
\ar[r]^{id} \ar[u]^{id} & M_{0,\alpha} \ar[r]^{id} \ar[u]^{id} &
M_{0,\alpha+1} \ar[r]^{id} \ar[u]^{id} & M_{0,\alpha^*}
\ar[r]^{id} \ar[u]^{id} & M_{0,\alpha^*+1}=M^+ \ar[u]^{id}}
\end{displaymath}

We cannot succeed, because if we succeed then letting
$M^*:=\bigcup \{M_{\beta,\alpha^*}:\beta<\lambda^+\}$, the
sequences $\langle M^* \bigcap
M_{\beta,\alpha^*+1}:\beta<\lambda^+ \rangle,\ \langle
M_{\beta,\alpha^*}:\beta<\lambda^+ \rangle$ are representations of
$M^*$. So for a club of $\beta<\lambda^+$ $M^* \bigcap
M_{\beta,\alpha^*+1}=M_{\beta,\alpha^*}$. Take such a $\beta$.
$$M_{\beta,\alpha^*} \subseteq M_{\beta+1,\alpha^*} \bigcap
M_{\beta,\alpha^*+1} \subseteq M^* \bigcap
M_{\beta,\alpha^*+1}=M_{\beta,\alpha^*}$$, hence this is an
equivalences chain, in contradiction to condition (v).

Where will we get stuck? Obviously we will not get stuck at
$\beta=0$. For $\beta$ limit we define $M_{\beta,\alpha}=\bigcup
\{M_{\gamma,\alpha}:\gamma<\beta\}$ and by smoothness (Definition
\ref{1.1}.d) $\langle M_{\beta,\alpha}:\alpha \leq \alpha^*+1
\rangle$ is increasing. It remains to get stuck at some successor
ordinal. Let $\beta$ be the first ordinal such that there is no
$\langle M_{\beta+1,\alpha}:\alpha \leq \alpha^*+1 \rangle$ which
satisfies clauses (i)-(vii).

\case{Case a:} $M_{\beta,\alpha^*+1}=M_{\beta,\alpha^*}$. We
define $N_\alpha:=M_{\beta,\alpha}$ for $\alpha \leq \alpha^*+1$.
$M^+=M_{0,\alpha^*+1} \preceq
M_{\beta,\alpha^*+1}=M_{\beta,\alpha^*}=N_{\alpha^*}$. So the
proposition is proved.

\case{Case b:} $M_{\beta,\alpha^*} \prec M_{\beta,\alpha^*+1}$. We
prove that we will not get stuck. By the densite of basic types
(in Definition \ref{definition of a good frame minus stability})
there is $a \in M_{\beta,\alpha^*+1}$ such that
$(M_{\beta,\alpha^*},M_{\beta,\alpha^*+1},a) \in
S^{bs}(M_{\beta,\alpha^*})$. Denote
$p_\beta:=tp(a,M_{\beta,\alpha^*},M_{\beta,\alpha^*+1})$. By
Proposition \ref{adding a type to an independent sequence} there
is a sequence $\langle M_{\beta+1,\alpha}:\alpha \leq \alpha^*
\rangle$ which satisfies clauses $(i),(iii)$ such that $p$ is
realized in $M_{\beta+1,\alpha^*}$. Take $b \in
M_{\beta+1,\alpha^*}$ which realizes $p$. By the definition of a
type without loss of generality $a=b$.
\end{proof}

\begin{proposition}\label{improving independence's definition
proposition} \mbox{}\\
(a) if the set $J$ is independent in
$(M,A,N)$ then there is an independent sequence $\langle
M_\alpha,a_\alpha:\alpha<\alpha^* \rangle ^\frown \langle
M_{\alpha^*} \rangle$ over $M$ such that $A \subseteq M_0$, $J=\{a_\alpha:\alpha<\alpha^*\}$ and $N \preceq M_{\alpha^*}$.\\
(b) if the sequence $\langle a_\alpha:\alpha<\alpha^* \rangle$ is
independent in $(M,A,N)$ then there is an increasing continuous
sequence $\langle M_\alpha:\alpha \leq \alpha^* \rangle$ such that
$A \subseteq M_0$, the sequence $\langle
M_\alpha,a_\alpha:\alpha<\alpha^* \rangle ^\frown \langle
M_{\alpha^*} \rangle$ is independent over $M$ and $N \preceq M_{\alpha^*}$.\\

\end{proposition}

\begin{proof}
By Proposition \ref{increasing the big model}.
\end{proof}

\begin{proposition}[2-concatenation]\label{2-concatenation}
If
\begin{enumerate}
\item $M \preceq M_0 \preceq M_1 \preceq M_2$. \item $\langle
a_{0,\beta}:\beta<\beta_0 \rangle$ is independent in
$(M,M_0,M_1)$. \item $\langle a_{1,\beta}:\beta<\beta_1 \rangle$
is independent in $(M,M_1,M_2)$.
\end{enumerate}
Then $\langle a_{0,\beta}:\beta<\beta_0 \rangle ^\frown \langle
a_{1,\beta}:\beta<\beta_1 \rangle$ is independent in
$(M,M_0,M_2)$.
\end{proposition}

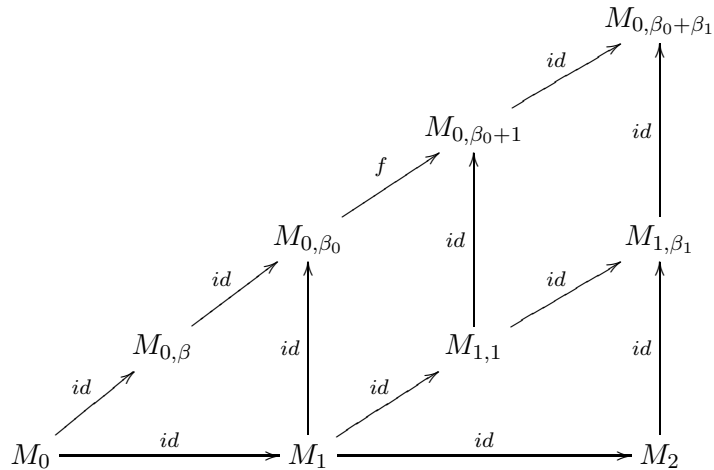
\begin{proof}
\begin{displaymath}
\xymatrix{&&&&M_{0,\beta_0+\beta_1}\\
&&& M_{0,\beta_0+1} \ar[ru]^{id}\\
&& M_{0,\beta_0} \ar[ru]^{f} && M_{1,\beta_1} \ar[uu]^{id}\\
& M_{0,\beta} \ar[ru]^{id} && M_{1,1} \ar[uu]^{id} \ar[ru]^{id}\\
M_0 \ar[ru]^{id} \ar[rr]^{id} && M_1 \ar[ru]^{id} \ar[uu]^{id}
\ar[rr]^{id} && M_2 \ar[uu]^{id} }
\end{displaymath}
By clause 2 and Proposition \ref{improving independence's
definition proposition}, there is a sequence $\langle
M_{0,\beta}:\beta \leq \beta_0 \rangle$ such that the sequence
$\langle M_{0,\beta},a_{0,\beta}:\beta<\beta_0 \rangle ^\frown
\langle M_{0,\beta_0} \rangle$ is independent over $M$ and $M_1
\preceq M_{0,\beta_0}$. Similarly by clause 3 and Proposition
\ref{improving independence's definition proposition} there is a
sequence $\langle M_{1,\beta}:\beta \leq \beta_1 \rangle$ such
that the sequence $\langle M_{1,\beta},a_{1,\beta}:\beta<\beta_1
\rangle ^\frown \langle M_{1,\beta_1} \rangle$ is independent over
$M$ and $M_2 \preceq M_{1,\beta_1}$. By Fact \ref{3.3}, for some
$f,\langle M_{0,\beta_0+\beta}:\beta \leq \beta_1 \rangle $ the
following hold:
\begin{enumerate}[(a)]
\item $f:M_{0,\beta_0} \to M_{0,\beta_0+1}$ is an embedding over
$M_1$. \item $\langle M_{0,\beta_0+\beta}:\beta \leq \beta_1
\rangle$ is an increasing continuous sequence. \item For $\beta
\leq \beta_1$, $M_{1,\beta} \preceq M_{0,\beta_0+\beta}$. \item
For $\beta<\beta_1$,
$tp(a_{1,\beta},M_{0,\beta_0+\beta},M_{0,\beta_0+\beta+1})$ does
not fork over $M_{1,\beta}$.
\end{enumerate}
Since for $\beta<\beta_0$, $a_{0,\beta} \in M_1$ it follows that
$f(a_{0,\beta})=a_{0,\beta}$. Therefore the sequence $\langle
f[M_{0,\beta}],a_{0,\beta}:\beta<\beta_0 \rangle ^\frown \langle
M_{0,\beta_0+\beta},a_{1,\beta}:\beta<\beta_1 \rangle ^\frown
\langle M_{0,\beta_0+\beta_1} \rangle$ is independent over $M$.
But $M_2 \preceq M_{1,\beta_1} \preceq M_{0,\beta_0+\beta_1}$.
\end{proof}

\begin{theorem}\label{change order of a finite set}\label{middle to first}
\mbox{}\\
(a) If \emph{a finite set} $J$ is independent in $(M,M_0,M^+)$ and
$\{a_\alpha:\alpha<\alpha^*\}$ is an enumeration of $J$ without
repetitions then the \emph{sequence} $\langle
a_\alpha:\alpha<\alpha^* \rangle$ is independent in
$(M,M_0,M^+)$.\\
(b) If the sequence $\langle M_{\alpha},a_\alpha:\alpha \leq \beta
\rangle^\frown \langle M_{\beta+1} \rangle$ is independent over
$M$, Then the sequence $\langle a_\beta \rangle ^\frown \langle
a_\alpha:\alpha<\beta \rangle$ is independent in
$(M,M_0,M_{\beta+1})$.\\
(c) If the sequence $\langle a_\alpha:\alpha<\beta+1 \rangle$ is
independent in $(M,M_0,N)$, then the sequence $\langle a_\beta
\rangle ^\frown \langle a_\alpha:\alpha<\beta \rangle$ is
independent in $(M,M_0,N)$.\\
(d) If the sequence $\langle M_{\alpha},a_\alpha:\alpha<\alpha^*
\rangle^\frown \langle M_{\alpha^*} \rangle$ is independent over
$M$ and $\beta_0<\beta_1<\alpha^*$ then the sequence $\langle
a_{\beta_1} \rangle ^\frown \langle a_\alpha:\alpha \in
[\beta_0,\beta_1) \rangle ^\frown \langle a_\alpha:\alpha \in
(\beta_1,\alpha^*) \rangle$ is independent in
$(M,M_{\beta_0},M_{\beta_1+1})$. So the sequence $\langle
a_\alpha:\alpha<\beta_0 \rangle ^\frown \langle a_{\beta_1}
\rangle ^\frown \langle a_\alpha:\alpha \in [\beta_0,\beta_1)
\rangle ^\frown \langle a_\alpha:\alpha \in (\beta_1,\alpha^*)
\rangle$ is independent in $(M,M_0,M_{\alpha^*})$.
\end{theorem}

\begin{proof}
\mbox{}

(a) Follows by item d.

(b) By Fact \ref{3.3} for some $f,\langle N_\alpha:\alpha \leq
\beta \rangle$ the following hold:
\begin{enumerate}
\item $f$ is an isomorphism of $M_{\beta+1}$ to $N_1$ over $M_0$.
\item $N_0=f[M_{\beta+1}]$. \item for $\alpha \leq \beta$,
$M_\alpha \preceq N_\alpha$. \item
$tp(f(a_\beta),M_\beta,N_\beta)$ does not fork over $M_0$. \item
For $\alpha<\beta$, $tp(a_\alpha,N_\alpha,N_{\alpha+1})$ does not
fork over $M_\alpha$.
\end{enumerate}

\begin{displaymath}
\xymatrix{a_\beta \in M_{\beta+1} \ar[r]^{f} & N_1 \ar[r]^{id} & N_\beta \ar[r]^{g} & N_{\beta+1}\\
M_0 \ar[r]^{id} \ar[u]^{id} & M_1 \ar[r]^{id} \ar[u]^{id} &
M_\beta \ar[r]^{id} \ar[u]^{id} & M_{\beta+1} \ar[u]^{id} \ni
a_\beta}
\end{displaymath}

By assumption, $tp(a_\alpha,M_\alpha,M_{\alpha+1})$ does not fork
over $M$. Hence by Proposition \ref{transitivity} and clause 5
$tp(a_\alpha,N_\alpha,N_{\alpha+1})$ does not fork over $M$.
Therefore:

(6) The sequence $\langle N_\alpha,a_\alpha:\alpha<\beta \rangle
^\frown \langle N_\beta \rangle$ is independent in
$(M,N_0,N_\beta)$.

Since $f$ is an isomorphism over $M_0$, by clause 2
$tp(f(a_\beta),M_0,N_0)=tp(a_\beta,M_0,M_{\beta+1})$. By
assumption, $tp(a_\beta,M_\beta,M_{\beta+1})$ does not fork over
$M_0$. So by clause 4 and Definition \ref{definition of a good
frame minus stability}.d (uniqueness)
$tp(f(a_\beta),M_\beta,\allowbreak
N_\beta)=tp(a_\beta,M_\beta,M_{\beta+1})$. hence there is an
amalgamation $(g,id_{M_{\beta+1}},N_{\beta+1})$ of
$N_\beta,M_{\beta+1}$ over $M_\beta$ such that
$g(f(a_\beta))=a_\beta$. Since $g$ is an embedding over $M_\beta$,
it follows that $g(a_\alpha)=a_\alpha$. Therefore by clause 6 the
sequence $\langle M_0,a_\beta \rangle ^\frown \langle
g[N_\alpha],a_\alpha:\alpha<\beta \rangle ^\frown \langle
g[N_\beta] \rangle$ is independent over $M$.

(c) By item b and Proposition \ref{improving independence's
definition proposition}.

(d) By definition the sequence $\langle a_\alpha:\alpha \in
[\beta_0,\beta_1] \rangle$ is independent in
$(M,M_{\beta_0},M_{\beta_1+1})$. So by item c, the sequence
$\langle a_{\beta_1} \rangle ^\frown \langle a_\alpha:\alpha \in
[\beta_0,\beta_1) \rangle$ is independent in
$(M,M_{\beta_0},M_{\beta_1+1})$. Now use Proposition
\ref{2-concatenation} twice.
\end{proof}

The following proposition is similar to Claim 5.13 of
\cite{sh705}.
\begin{proposition}[The exchange proposition]\label{exchange}
Let $J$ be an independent set in $(M,M_0,N)$.
\begin{enumerate}
\item If $tp(a,M,N) \in S^{bs}(M)$, then there is a finite subset
$J^* \subset J$ such that $J \bigcup \{a\}-J^*$ is independent in
$(M,M_0,N)$. \item If $a \in N$ then for some models $M',N'$ and a
subset $J^* \subseteq J$ the following hold:
\begin{enumerate}
\item $|J^*|<\aleph_0$. \item $M \bigcup \{a\} \subseteq M'
\preceq N'$. \item $N \preceq N'$. \item $J-J^*$ is independent in
$(M,M',N')$.
\end{enumerate}
\end{enumerate}
\end{proposition}

\begin{proof}
By assumption, for some independent sequence $\langle
M_\alpha,a_\alpha:\alpha<\alpha^* \rangle ^\frown \allowbreak
\langle M_{\alpha^*} \rangle$ over $M$
$\{a_\alpha:\alpha<\alpha^*\}=J$ and there is $N^+$ such that $N
\preceq N^+$ and $M_{\alpha^*} \preceq N^+$. We prove the
proposition by induction on $\alpha^*$.

(1) Assume that $\alpha^*=0$. So $J=\emptyset$ and $\{a\}$ is
independent in $(M,M_0,N)$. For $\alpha^*=\beta+1$ we can subtract
$a_\beta$. Assume $\alpha^*$ is limit. By Proposition
\ref{preparation exchange} without loss of generality
$N^+=M_{\alpha^*}$ so $a \in M_\alpha$ for some $\alpha<\alpha^*$.
By the induction hypotheses, there is a finite subset $J^*
\subseteq \{a_\gamma:\gamma<\alpha \}$, such that
$\{a_\gamma:\gamma<\alpha \} \bigcup \{a\}-J^*$ is independent in
$(M,M_0,M_\alpha)$. But $\{a_\gamma:\alpha \leq \gamma<\alpha^*
\}$ is independent in $(M,M_\alpha,M_{\alpha^*})$. Therefore by
Proposition \ref{2-concatenation} $\{a_\alpha:\alpha<\alpha^*\}
\bigcup \{a\}-J^*$ is independent in $(M,M_0,M_{\alpha^*})$.

(2) For $\alpha^*=0$ we define $M':=N,\ N':=N,\ J^*:=\emptyset$
($J-J^*=\emptyset$, so clause d is not relevant). For
$\alpha^*=\beta+1$ we can subtract $a_\beta$: By the induction
hypotheses for some models $M',N'$ and a subset $J^* \subseteq
\{a_\alpha:\alpha<\beta\}$ the following hold:
\begin{enumerate}[(a)]
\item $|J^*|<\aleph_0$. \item $M \bigcup \{a\} \subseteq M'
\preceq N'$. \item $N \preceq N'$. \item
$\{a_\alpha:\alpha<\beta\}-J^*$ is independent in $(M,M',N')$.
\end{enumerate}
Now $M',N',J^* \bigcup \{a_\beta\}$ are as required.

Assume that $\alpha^*$ is a limit ordinal. By Proposition
\ref{preparation exchange} without loss of generality
$N=M_{\alpha^*}$ so $a \in M_\alpha$ for some $\alpha<\alpha^*$.
So by the induction hypotheses, there are $M',N',J^*$ such that
the following hold:
\begin{enumerate}[(a)]
\item $|J^*|<\aleph_0$. \item $M \bigcup \{a\} \subseteq M'
\preceq N'$. \item $M_\alpha \preceq N'$. \item
$\{a_\beta:\beta<\alpha\}-J^*$ is independent in $(M,M',N')$.
\end{enumerate}

By Proposition \ref{amalgamation of a model and an independent
sequence} we can find an amalgamation $(id_{M_{\alpha^*}},f,N^+)$
of $N',M_{\alpha^*}$ over $M_\alpha$ such that the set
$\{a_\beta:\alpha \leq \beta<\alpha^* \}$ is independent in
$(M,f[N'],\allowbreak N^+)$.

\begin{displaymath}
\xymatrix{ a \in M' \ar[r]^{id} & N'  \ar[r]^{f} & N^+ \\
M_0 \ar[r]^{id} \ar[u]^{id} & M_\alpha \ar[r]^{id} \ar[u]^{id} &
M_{\alpha^*}=N \ar[u]^{id}}
\end{displaymath}

Since $f$ is an isomorphism over $M_\alpha$,
$\{a_\beta:\beta<\alpha\}-J^*$ is independent in
$(M,f[M'],f[N'])$. Now by Proposition \ref{2-concatenation}
$J-J^*$ is independent in $(M,f[M'],N^+)$. Since $a \in M_\alpha$,
it follows that $a=f(a) \in f[M']$. Therefore $f[M'],N^+,J^*$ are
as required.
\end{proof}

\begin{definition}
We say that \emph{the independence relation has continuity} when
the following holds: If
\begin{enumerate}
\item $\delta$ is a limit ordinal. \item $\langle M_\alpha:\alpha
\leq \delta \rangle$ is an increasing continuous sequence. \item
For $\alpha<\delta$ the set $J$ is independent in
$(M,M_\alpha,M^+)$.
\end{enumerate}
Then $J$ is independent in $(M,M_\delta,M^+)$.
\end{definition}

\begin{definition}
We say that \emph{the finite independence relation has continuity}
when the following holds: If
\begin{enumerate}
\item $\delta$ is a limit ordinal. \item $\langle M_\alpha:\alpha
\leq \delta \rangle$ is an increasing continuous sequence. \item
For $\alpha<\delta$ the set $J$ is finitely independent in
$(M,M_\alpha,M^+)$.
\end{enumerate}
Then $J$ is finitely independent in $(M,M_\delta,M^+)$.
\end{definition}

\begin{proposition}\label{dimension proposition}
Assume: \begin{enumerate} \item $P \subseteq S^{bs}(M_0)$. \item
$J_1,J_2$ are maximal sets in \{$J:J$ is independent in
$(M,M_0,N)$ and $a \in J \Rightarrow tp(a,M_0,N) \in P$\} \item
The independence relation has continuity or $J_1$ is finite or
$J_2$ is finite.
\end{enumerate}
Then $|J_1|=|J_2|$ or they both finite.
\end{proposition}

\begin{proof}
Towards a contradiction assume that $|J_1|<|J_2| \geq \aleph_0$.
Let $\{a_\alpha:\alpha<|J_1|\}$ be an enumeration of $J_1$ without
repetitions. We define $\mu:=|J_1|$. We choose by
 induction on $\alpha \leq \mu$ a triple $(M'_\alpha,N'_\alpha,J^*_\alpha)$ such that:
 \begin{enumerate} \item
$\langle M'_\alpha:\alpha \leq |J_1| \rangle$ is an increasing
continuous sequence of models. \item $\langle N'_\alpha:\alpha
\leq |J_1| \rangle$ is an increasing continuous sequence of
models. \item For $\alpha \leq |J_1|$, $M'_\alpha \preceq
N'_\alpha$. \item $M'_0=M_0$. \item $N'_0=N$. \item $a_\alpha \in
M'_{\alpha+1}$. \item $J^*_\alpha$ is a finite subset of
$J_2-\bigcup \{J^*_\beta:\beta<\alpha\}$. \item
$J_2-\bigcup\{J^*_\beta:\beta \leq \alpha\}$ is independent in
$(M,M'_\alpha,N'_\alpha)$.
 \end{enumerate}
Why can we carry out this induction? For $\alpha=0$ we choose
$M'_0:=M_0,N'_0=N,J^*_0:=\emptyset$.

In the $\alpha+1$ step, we use Proposition \ref{exchange}.2 (the
exchange proposition). How? We substitute $J_2-\{J^*_\beta:\beta
\leq \alpha\},M,M'_\alpha,N'_\alpha,a_\alpha$ instead of
$J,M,M_0,N,a$. By clause 8, $J_2-\{J^*_\beta:\beta \leq \alpha\}$
is independent in $(M,M'_\alpha,N'_\alpha)$. $a_\alpha \in J_1$,
so by assumption 2, $a_\alpha \in N$. By clauses 2,5 $N \subseteq
N'_\alpha$, so $a_\alpha \in N'_\alpha$. Therefore by the exchange
proposition, for some $M',N'$ and a subset $J^* \subseteq
J_2-\bigcup\{J^*_\beta:\beta \leq \alpha\}$ the following hold:
\begin{enumerate}
\item $|J^*|<\aleph_0$. \item $M'_\alpha \bigcup \{a_\alpha\}
\subseteq M' \preceq N'$. \item $N'_\alpha \preceq N'$. \item
$J_2-\bigcup\{J^*_\beta:\beta \leq \alpha\}-J^*$ is independent in
$(M,M',N')$.
\end{enumerate}
Now we define $M'_{\alpha+1}:=M',\ N'_{\alpha+1}:=N',\
J^*_{\alpha+1}:=J^*$.

For limit $\alpha$ we define $M'_\alpha:=\bigcup
\{M'_\beta:\beta<\alpha\},\ N'_\alpha:=\bigcup
\{N'_\beta:\beta<\alpha\}$ and $J^*_\alpha:=\emptyset$. For every
$\beta<\alpha$ by clause 8 the set $J_2-\bigcup
\{J^*_\gamma:\gamma \leq \beta\}$ is independent in
$(M,M'_\beta,N'_\beta)$. So $J_2-\bigcup
\{J^*_\beta:\beta<\alpha\}$ is independent in
$(M,M'_\beta,N'_\beta)$. Hence by the continuity $J_2-\bigcup
\{J^*_\beta:\beta<\alpha\}$ is independent in
$(M,M'_\alpha,N'_\alpha)$.

By assumption, $J_1$ is independent in $(M,M_0,N)$. By clauses 6,1
$J_1 \subseteq M'_\mu$. But $M'_\mu \preceq N'_\mu \succeq N_0$.
Hence $J_1$ is independent in $(M,M_0,M'_\mu)$. By clause 8
$J_2-\bigcup\{J^*_\beta:\beta \leq \mu\}$ is independent in
$(M,M'_\mu,N'_\mu)$. Therefore by Proposition
\ref{2-concatenation} $J_1 \bigcup (J_2-\bigcup\{J^*_\beta:\beta
\leq \mu\})$ is independent in $(M,M_0,N_\mu')$, hence in
$(M,M_0,N)$. But it contradicts the maximality of $J_1$! ($J_1
\bigcup J_2-\bigcup\{J^*_\beta:\beta \leq \mu\} \neq J_1$, because
$|J_1|<|J_2|>\bigcup\{J^*_\beta:\beta \leq \mu\}$).
\end{proof}

\begin{definition}\label{dimension}
Suppose $M \preceq_{\frak{k}_\lambda} N$ and let $P \subseteq
S^{bs}(M)$.
$$dim(P,N):=Min\left\{\begin{array}{c|l}{|J|}& a \in J \Rightarrow tp(a,M,N) \in P \\
&J \text{ is independent in } (M,N)
\\ &J \text{ is maximal under the previous conditions} \\
 \end{array}\right\}.$$
For $p \in S^{bs}(M)$ we define $dim(p,N)=dim(\{p\},N)$.
\end{definition}

\section{Finite Character}
In this section we use uniqueness triples. In Definition
\ref{definition uniqueness triples}=4.3 we define uniqueness
triples. In Definition \ref{4.4}=4.4 we define an independent
sequence by $K^{3,uq}$, the class of uniqueness triples.
Proposition \ref{4.9}=4.9.c asserts that the existence property
for $K^{3,uq}$ implies that any independent sequence is
independent by $K^{3,uq}$. Proposition \ref{4.10}=4.10 explain the
advantage of uniqueness triples. Proposition
\ref{complete...}=4.11 is similar to Proposition
\ref{complete...}=4.10, but we replace here independent sequence
by independent set. Theorem \ref{complete...}=4.12 is the main
theorem of the paper. It asserts that independence satisfies the
expected properties.

\begin{definition}\label{equivalent amalgamations}
Suppose
\begin{enumerate} \item
$M_0 \preceq_\frak{s}M_1 \wedge M_0 \preceq_\frak{s}M_2$. \item
For $x=a,b$, $(f_1^x,f_2^x,M_3^x)$ is an amalgamation of $M_1,M_2$
over $M_0$.
\end{enumerate}
$(f_1^a,f_2^a,M_3^a),(f_1^b,f_2^b,M_3^b)$ are said to be
\emph{equivalent} over $M_0$ if there are $f^a,f^b,M_3$ such that
the following diagram commutes:
\begin{displaymath}
\xymatrix{&M_3^b \ar[r]^{f^b} & M_3\\
M_1 \ar[ru]^{f_1^b} \ar[rr]^{f_1^a} && M_3^a \ar[u]_{f^a}\\
M_0 \ar[u]^{id_{M_0}} \ar[r]_{id_{M_0}} & M_2 \ar[ru]_{f_2^a}
\ar[uu]^{f_2^b}}
\end{displaymath}
We denote the relation ``to be equivalent over $M_0$'' between
amalgamations over $M_0$, by $E_{M_0}$.
\end{definition}

\begin{proposition}[\cite{jrsh875}]
The relation $E_{M_0}$ is an equivalence relation.
\end{proposition}
We will not give here a proof, as we do not use this proposition.

\begin{definition}\label{definition uniqueness triples}
$K^{3,uq}$ is the set of triples $(M_{0,0},M_{1,0},a) \in
S^{bs}(M_{0,0})$ such that for every model $M_{0,1} \succ M_{0,0}$
there is a unique amalgamation (up to $E_{M_{0,0}}$)
$(M_{1,1},f_{1,0},f_{0,1})$ of $M_{1,0},M_{0,1}$ over $M_{0,0}$,
such that $f_{1,0}(tp(a,M_{1,0},\allowbreak M_{0,0}))$ does not
fork over $M_{0,0}$. A \emph{uniqueness triple} is a triple in
$\frak{k}^{3,uq}$.
\end{definition}

It is reasonable to assume that the existence property is
satisfied by $K^{3,uq}$, because if $N \in K_\lambda$, $|S(N)|
\leq \lambda^+$, and the existence property is not satisfied by
$K^{3,uq}$, then by the last corollary in section 4 of
\cite{jrsh875} there are $2^{\lambda^{+2}}$ non isomorphic models
in $K_{\lambda^{+2}}$ assuming weak set-theoretic assumptions.
Note that by Claims 1.18,1.20 of \cite{she46} we have the
following fact:
\begin{fact}
Assume $\lambda \geq \aleph_0$ and $\frak{k}_\lambda$ has
amalgamation. If $N \in K_\lambda$ and $|S(N)|>\lambda^+$, then
$|\{N' \in K_\lambda:N \prec N'\}/\cong|>\lambda^+$.
\end{fact}

\begin{remark}\label{uniqueness is closed under isomorphism}
$K^{3,uq}$ is closed under isomorphisms: If $(M_0,M_1,a) \in
\frak{k}^{3,uq}$ and $f:M_1 \to M_1^*$ is an isomorphism, then
$(f[M_0],f[M_1],f(a)) \in \frak{k}^{3,uq}$.
\end{remark}

\begin{proposition}\label{non-forking by uniqueness}
If
\begin{enumerate}
\item $M_{0,0} \preceq M_{1,0} \preceq M_{1,1}$. \item $M_{0,0}
\preceq M_{0,1} \preceq M_{1,1}$. \item $(M_{0,0},M_{1,0},a) \in
K^{3,uq}$. \item $tp(a,M_{0,1},M_{1,1})$ does not fork over
$M_{0,0}$. \item $tp(b,M_{0,0},M_{0,1}) \in S^{bs}(M_{0,0})$.
\end{enumerate}
Then $tp(b,M_{1,0},M_{1,1})$ does not fork over $M_{0,0}$.
\end{proposition}

\begin{proof}
We have two amalgamations of $M_{1,0},M_{0,1}$ over $M_{0,0}$ such
that the image of $tp(a,M_{0,1},M_{1,1})$ does not fork over
$M_{0,0}$: One is the amalgamation
$(M_{1,1},id_{M_{1,0}},id_{M_{0,1}})$. The second exists by Fact
\ref{axiom i is abandonment}. So by Definition \ref{definition
uniqueness triples} they are equivalent. So as in the second
amalgamation the types do not fork, so does in the first.
\end{proof}

\begin{proposition}\label{monotonicity of uniqueness}
If $(M_0,M_1,a) \in \frak{k}^{3,uq}$, $(M_0,M_2,a) \in
\frak{k}^{3,bs}$ and $M_0 \preceq M_2 \preceq M_1$ then
$(M_0,M_2,a) \in \frak{k}^{3,uq}$.
\end{proposition}

\begin{proof}
Easy.
\end{proof}

\begin{definition}\label{4.4}
$\langle M_\alpha,a_\alpha:\alpha<\alpha^* \rangle ^\frown \langle
M_{\alpha^*} \rangle$ is said to be independent over $M$ \emph{by
uniqueness triples} when we add in Definition \ref{definition of
independence}: $(M_\alpha,M_{\alpha+1},a_\alpha) \in K^{3,uq}$.
Similarly for independent set.
\end{definition}

Roughly speaking the following proposition says that: Suppose that
we have an independent sequence.
\begin{enumerate}[(a)]
\item One can replace the first triple in the sequence by any
other triple with the same type. \item Like item 1, but for all
the triples simultaneously, i.e. if someone chooses the triples
(up to isomorphisms), we will still be able to find a witness for
independence. \item If the existence property is satisfied by
$K^{3,uq}$, then the sequence is independent by uniqueness
triples.
\end{enumerate}

\begin{definition}\label{the independence game}
Suppose $\langle M_{0,\alpha},a_\alpha:\alpha<\alpha^* \rangle
^\frown \langle M_{0,\alpha^*} \rangle$ is an independent sequence
over $M$. \emph{The independence game} for the sequence $\langle
M_{0,\alpha},a_\alpha:\alpha<\alpha^* \rangle ^\frown \langle
M_{0,\alpha^*} \rangle$ over $M$ is a two-player game that lasts
$\alpha^*+1$ moves.
\begin{displaymath}
\xymatrix{b_2 \in M_{3,0} \ar[rrr]^{f_3} &&& M_{3,3}\\
b_1 \in M_{2,0} \ar[u]^{id} \ar[rr]^{f_2} &&
M_{2,2} \ar[r]^{id}  & M_{2,3} \ar[u]^{id}\\
b_0 \in M_{1,0} \ar[u]^{id} \ar[r]^{f_1} & M_{1,1} \ar[r]^{id} &
M_{1,2} \ar[u]^{id} \ar[r]^{id}
 & M_{1,3} \ar[u]^{id}\\
M_{0,0} \ar[r]^{f_0} \ar[u]^{id} & a_0 \in M_{0,1} \ar[u]^{id}
\ar[r]^{id} & a_1 \in M_{0,2} \ar[u]^{id} \ar[r]^{id} & a_2 \in
M_{0,3} \ar[u]^{id}}
\end{displaymath}
The $\alpha+1$ move: Player 1 chooses $M_{\alpha+1,0},b_\alpha$,
such that $tp(b_\alpha,M_{\alpha,0},M_{\alpha+1,0})=\allowbreak
tp(a_\alpha,M_{\alpha,0},M_{\alpha,\alpha+1})$. Then player 2
chooses $f_{\alpha+1}$ and a sequence $\langle
M_{\alpha+1,\beta}:\beta \in [\alpha+1,\alpha^*] \rangle$ (for
$\beta \in (0,\alpha+1)$, $M_{\alpha+1,\beta}$ is not defined)
such that:
\begin{enumerate}
\item $f_{\alpha+1}$ is an injection with domain $M_{\alpha+1,0}$.
\item $f_\alpha \subset f_{\alpha+1}$. \item
$f_{\alpha+1}(b_\alpha)=a_\alpha$. \item For $\beta \in
(\alpha,\alpha^*)$
$tp(a_\beta,M_{\alpha+1,\beta},M_{\alpha+1,\beta+1})$ does not
fork over $M_{\alpha,\beta}$.
\end{enumerate}
For $\alpha=0$ we define $f_0:=id_{M_{0,0}}$. For $\alpha$ limit,
in the $\alpha$ move, we define $M_{\alpha,0}:=\bigcup
\{M_{_\beta,0}:\gamma<\alpha \},\ f_\alpha:=\bigcup
\{f_\gamma:\gamma<\alpha \}$ and for $\beta \in [\alpha,\alpha^*]$
$M_{\alpha,\beta}:=\bigcup \{M_{\gamma,\beta}:\gamma<\alpha \}$.
So for $\alpha=0$ or limit, in the $\alpha$ move the players do
not have any choice. Player 2 wins if he has always a legal move
(so in this case, in the end of the game, the sequence $\langle
M_{\alpha,\alpha},a_\alpha:\alpha<\alpha^* \rangle ^\frown \langle
M_{\alpha^*,\alpha^*} \rangle$ is independent over $M$).
\end{definition}

\begin{proposition} \label{replacetriples}\label{4.9}
Suppose: The sequence $\langle
M_{0,\alpha},a_\alpha:\alpha<\alpha^* \rangle^\frown \langle
M_{0,\alpha^*} \rangle$ is independent over $M$.

(a) If $tp(b,M_{0,0},M_{1,0})=tp(a_0,M_{0,0},M_{0,1})$, Then for
some sequence $\langle M_{1,\alpha}:0<\alpha \leq \alpha^*
\rangle$ and $f$ the following hold:
\begin{enumerate}
\item $\langle M_{1,\alpha}:\alpha \leq \alpha^* \rangle$ is an
increasing continuous sequence of models in $K_\lambda$. \item
$M_{0,\alpha} \preceq M_{1,\alpha}$ for each $\alpha \leq
\alpha^*$. \item $0<\alpha<\alpha^* \Rightarrow a_\alpha \in
M_{1,\alpha+1}-M_{1,\alpha}$. \item $0<\alpha<\alpha^* \Rightarrow
tp(a_\alpha,M_{1,\alpha},M_{1,\alpha+1})$ does not fork over $M$.
\item $f:M_{1,0} \to M_{1,1}$ is an embedding over $M_{0,0}$, and
$f(b)=a_0$. \item The sequence $\langle M_{0,0},a_0,f[M_{1,0}],a_1
\rangle ^\frown \langle M_{1,\alpha},a_\alpha:\allowbreak 1<\alpha
<\alpha^*
 \rangle  ^\frown \langle M_{1,\alpha^*} \rangle$ is
independent over $M$.
\begin{displaymath}
\xymatrix{b \in M_{1,0} \ar[r]^{f} & M_{1,1} \ar[r] & a_1 \in
M_{1,2} \ar[r]  & a_2 \in M_{1,3}\\
M_{0,0} \ar[r] \ar[u] & a_0 \in M_{0,1} \ar[u] \ar[r] & a_1 \in
M_{0,2} \ar[u] \ar[r] & a_2 \in M_{0,3} \ar[u]}
\end{displaymath}
\end{enumerate}
(b)  Player 2 has a winning strategy in the independence game.

(c) If the existence property is satisfied by $K^{3,uq}$, then the
sequence $\langle a_\alpha:\alpha<\alpha^* \rangle$ is independent
in $(M,M_0,M_{\alpha^*})$ by uniqueness triples.
\end{proposition}

\begin{proof} (a): As $tp(b,M_{0,0},M_{1,0})=tp(a_0,M_{0,0},M_{0,1})$, there are
$g,M_{1,1}^{temp}$ such that $M_{0,1} \preceq M_{1,1}^{temp},\
g:M_{1,0} \to M_{1,1}^{temp}$ is an embedding over $M_{0,0}$, and
$g(b)=a_0$. Now by Fact \ref{3.3} (amalgamation of a model and a
sequence), for some $h,\langle M_{1,\alpha}:2 \leq \alpha \leq
\alpha^* \rangle$ the following hold:
\begin{enumerate}
\item For $\alpha \in [2,\alpha^*]$ we have $M_{0,\alpha} \preceq
M_{1,\alpha}$. \item For $\alpha \in [1,\alpha^*)$
$tp(a_\alpha,M_{1,\alpha},M_{1,\alpha+1})$ does not fork over
$M_{0,\alpha}$. \item $h:M_{1,1}^{temp} \to M_{1,2}$ is an
embedding over $M_{0,1}$.
\end{enumerate}

Now define $f:=h \circ g$. Clearly $f(b)=a_0$ and $f$ fixes
$M_{0,0}$. Why is condition 4 satisfies? complete... Why is
condition 5 satisfies? complete...


$(a) \Rightarrow (b)$: By item a, player 2 has always a legal
move.
Now we have to prove that if player 2 wins the game, then (in the
end of the game) the sequence $\langle
M_{\beta,\beta},a_\beta:\beta<\alpha^* \rangle ^\frown \langle
M_{\alpha^*,\alpha^*} \rangle$ is independent over $M$.

First, why does the sequence $\langle
M_{\beta,\beta}:\beta<\alpha^* \rangle$ is continuous? Take
$\beta<\alpha^*$ limit and take an element $x \in
M_{\beta,\beta}$. There is an $\alpha<\beta$ such that $x \in
M_{\alpha,\beta}$. But $M_{\alpha,\beta}=\bigcup
\{M_{\alpha,\epsilon}:\epsilon<\beta\}$. so there is an
$\epsilon<\beta$ such that $x \in M_{\alpha,\epsilon}$. Define
$\gamma:=Max\{\alpha,\epsilon\}$. So $x \in M_{\gamma,\gamma}$ and
$\gamma<\beta$.

It remains to prove that for $\beta<\alpha^*$ the type
$tp(a_\beta,M_{\beta,\beta},M_{\beta+1,\beta+1})$ does not fork
over $M$. Let $\beta<\alpha^*$. For $\alpha<\beta$, by clause 4
$tp(a_\beta,M_{\alpha+1,\beta},M_{\alpha+1,\beta+1})$ does not
fork over $M_{\alpha,\beta}$. The sequence $\langle
M_{\alpha,\beta}:\alpha \leq \beta \rangle$ is increasing and
continuous. So $tp(a_\beta,M_{\beta,\beta},M_{\beta,\beta+1})$
does not fork over $M_{0,\beta}$. But by assumption, the sequence
$\langle M_{0,\alpha}:\alpha<\alpha^* \rangle ^\frown \langle
M_{0,\alpha^*} \rangle$ is independent over $M$. So
$tp(a_\beta,M_{0,\beta},M_{0,\beta+1})$ does not fork over $M$.
Hence by the transitivity (Proposition \ref{transitivity}),
$tp(a_\beta,M_{\beta,\beta},M_{\beta+1,\beta})$ does not fork over
$M$. But
$tp(a_\beta,M_{\beta,\beta},M_{\beta+1,\beta+1})=tp(a_\beta,M_{\beta,\beta},M_{\beta,\beta+1})$.

$b \Rightarrow c$: In the $\alpha+1$ step player 1 chooses a
triple in $K^{3,uq}$ and player 2 plays a winning strategy.
\end{proof}

\begin{proposition}\label{b_1-}
If
\begin{enumerate}
\item $\langle M_\alpha,a_\alpha:\alpha \leq \beta \rangle ^\frown
\langle M_{\beta+1} \rangle$ is independent over $M$. \item
$(M_0,N,a_\beta) \in K^{3,uq}$. \item $N \preceq M_{\beta+1}$.
\end{enumerate}
Then $\langle a_\alpha:\alpha<\beta \rangle$ is independent in
$(M,N,M_{\beta+1})$.
\end{proposition}

\begin{proof}
By Proposition \ref{amalgamation of a model and an independent
sequence} there is an amalgamation $(f,id_{M_\beta},M^{+})$ of
$N,M_\beta$ over $M_0$ such that $tp(f(a_\beta),M_\beta,M^{+})$
does not fork over $M_0$ and $\langle a_\alpha:\alpha<\beta
\rangle$ is independent in $(M,f[N],M^{+})$. Since
$tp(a_\beta,M_\beta,M_{\beta+1})$ does not fork over $M_0$, by the
definition of $K^{3,uq}$,
$(f,id_{M_\beta},M^+)E_{M_0}(id_{M_1},id_{M_\beta},M_{\beta+1})$.
So for some $g,M^{++}$ the following hold:
\begin{displaymath}
\xymatrix{& M^+ \ar[r]^{g} & M^{++}\\
N \ar[ru]^{f} \ar[rr]^{id} && M_{\beta+1} \ar[u]^{id} \\
M_0 \ar[u]^{id} \ar[r]^{id} & M_\beta \ar[uu]^{id} \ar[ru]^{id}}
\end{displaymath}

\begin{enumerate}
\item $M_{\beta+1} \preceq M^{++}$. \item $g:M^{+} \to M^{++}$ is
an embedding over $M_\beta$. \item $g \circ f=id_{N}$.
\end{enumerate}
So $\langle a_\alpha:\alpha<\beta \rangle$ is independent in
$(M,g[f[M_1]],g[M^+])$. Therefore it is independent in
$(M,N,M_{\beta+1})$.

\end{proof}

\begin{proposition}\label{b_1}
If
\begin{enumerate}
\item $J$ is finitely independent in $(M,M_0,M^+)$. \item $a \in
J$. \item $(M_0,M_1,a) \in K^{3,uq}$. \item $M_1 \preceq M^{+}$.
\end{enumerate}
Then $J-\{a\}$ is finitely independent in $(M,M_1,M^+)$.
\end{proposition}

\begin{proof}
let $J^*:=\{b_0,b_1,...,b_{n-1}\}$ be a finite subset of
$J-\{a\}$. By assumption, $J^* \bigcup \{a\}$ is independent in
$(M,M_0,M^+)$. By Theorem \ref{change order of a finite set}.a the
sequence $\langle b_0,b_1...b_{n-1} \rangle ^\frown \langle a
\rangle$ is independent in $(M,M_0,M^+)$. So by Proposition
\ref{b_1-} $\langle b_0,b_1...b_{n-1} \rangle$ is independent in
$(M,M_1,M^+)$.
\end{proof}

The main theorem:
\begin{theorem}\label{the main theorem}
Suppose $\frak{s}$ is a good $\lambda$-frame minus stability with
conjugation and the existence property is satisfied by $K^{3,uq}$.
\begin{enumerate}[(a)]
\item Independence is finitely independence.
\item If $J$ is independent in $(M,M_0,N)$ and the set $\{
a_\alpha:\alpha<\alpha^*\}$ is an enumeration of $J$ without
repetitions, then the sequence $\langle a_\alpha:\alpha<\alpha^*
\rangle$ is independent in $(M,M_0,N)$.
\item The independence relation has continuity.
\item If
\begin{enumerate}
\item $M \preceq N$. \item $P \subset S^{bs}(M)$. \item $J \subset
N$. \item $a \in J \Rightarrow tp(a,M,N) \in P$. \item $J$ is
independent in $(M,M,N)$. \item $J$ is maximal under the previous
conditions. \end{enumerate} Then $dim(P,N)=|J|$ or
$dim(P,N)+|J|<\aleph_0$.
\end{enumerate}
\end{theorem}
\begin{proof}
$(a \wedge c) \Rightarrow d$: By Proposition \ref{dimension
proposition}.

$a \Rightarrow c$: By Proposition \ref{finite continuity
proposition} below.

Therefore it is enough to prove items a,b.

\begin{proposition}\label{b implies d and c}
Items a,b of Theorem \ref{the main theorem} is implied by (*): If
\begin{enumerate}
\item $J$ is finitely independent in $(M,M_0,M^+)$. \item
$\{a_\alpha:\alpha<\alpha^*\} \subset J \subseteq M^+$. \item The
sequence $\langle M_\alpha,a_\alpha:\alpha<\alpha^* \rangle
^\frown \langle M_{\alpha^*} \rangle$ is independent over $M$ by
$K^{3,uq}$. \item $M_{\alpha^*} \preceq M^+$.
\end{enumerate}
Then $J-\{a_\alpha:\al \alpha<\al \alpha^*\}$ is finitely
independent in $(M,M_{\alpha^*},\al M^+)$.
\end{proposition}

\begin{proof}
(a) We prove the non-trivial direction. Assume that $J$ is
finitely independent in $(M,M_0,M^+)$. Denote $N_0:=M^+$. Take an
enumeration $\{a_\alpha:\alpha<|J|\}$ of $J$ without repetitions.
We choose by induction on $\gamma \in (0,|J|]$ a pair of models
$(M_\gamma,N_\gamma)$ such that:
\begin{enumerate}[*]
\item The sequence $\langle M_\alpha,a_\alpha:\alpha<\gamma
\rangle ^\frown \langle M_{\gamma} \rangle$ is independent over
$M$ by $K^{3,uq}$. \item $\langle N_\alpha:\alpha \leq \gamma
\rangle$ is increasing and continuous. \item For $\alpha \leq
\gamma$, $M_\alpha \preceq N_\alpha$.
\end{enumerate}

If we succeed to carry out this induction, then the sequence
$\langle M_\alpha,a_\alpha:\alpha<|J| \rangle$ is independent over
$M$, so the set $J$ is independent in $(M,M_0,M_{|J|})$ so in
$(M,M_0,M^+)$ (because $M^+ \preceq N_{|J|} \succeq M_{|J|}$).

Why can we carry out this induction? Assume that
$\gamma=\alpha^*+1$. We want to substitute $N_{\alpha^*}$ instead
of $M^+$ in assumption $(*)$ and to conclude that the set
$J-\{a_\alpha:\alpha<\alpha^*\}$ is finitely independent in
$(M,M_{\alpha^*},N_{\alpha^*})$. But why are the conditions of
$(*)$ satisfied?
\begin{enumerate}
\item Since $N=M^+$. \item By the definition of
$\{a_\alpha:\alpha<|J|\}$. \item By the induction hypothesis.
\item By the induction hypothesis $M_{\alpha^*} \preceq
N_{\alpha^*}$.
\end{enumerate}
Therefore we can conclude that the set
$J-\{a_\alpha:\alpha<\alpha^*\}$ is finitely independent in
$(M,M_{\alpha^*},N_{\alpha^*})$.

$\{a_{\alpha^*}\} \subseteq J-\{a_\alpha:\alpha<\alpha^*\}$ and it
is finite. So the set $\{a_{\alpha^*}\}$ is independent in
$(M,M_{\alpha^*},N_{\alpha^*})$, namely
$(M_{\alpha^*},N_{\alpha^*},a_{\alpha^*}) \in K^{3,bs}$ and
$tp(a_{\alpha^*},M_{\alpha^*},N_{\alpha^*})$ does not fork over
$M$. Since the existence property is satisfied by $K^{3,uq}$,
there are $M^{temp}_{\alpha^*+1},b_{\alpha^*}$ such that
$(M_{\alpha^*},M^{temp}_{\alpha^*+1},b_{\alpha^*}) \in K^{3,uq}$
and
$tp(b_{\alpha^*},M_{\alpha^*},M^{temp}_{\alpha^*+1})=tp(a_{\alpha^*},M_{\alpha^*},N_{\alpha^*})$.
So by the definition of a type (and Remark \ref{uniqueness is
closed under isomorphism}), for some models
$M_{\alpha^*+1},N_{\alpha^*+1}$ the following hold:
\begin{enumerate}[*]
\item $M_{\alpha^*} \preceq M_{\alpha^*+1} \preceq
N_{\alpha^*+1}$. \item $N_{\alpha^*} \preceq N_{\alpha^*+1}$.
\item $(M_{\alpha^*},M_{\alpha^*+1},a_{\alpha^*}) \in K^{3,uq}$.
\end{enumerate}

For limit $\gamma$ we take unions and use smoothness.

(b) By a similar proof.
\end{proof}

By \cite{jrsh875} (Definition 5.2 and Theorem 5.27):
\begin{fact} \label{jrsh875.5.15}
Suppose $\frak{s}$ is a good $\lambda$-frame minus stability with
conjugation and the existence property is satisfied by $K^{3,uq}$.
Then there is a (unique) relation $NF \subseteq \ ^4K_\lambda$
such that:
\begin{enumerate}[(a)]
\item If $NF(M_0,M_1,M_2,M_3) \ then \ n\in \{1,2\}\rightarrow
M_0\leq M_n\leq M_3 \ and \ M_1 \cap M_2=M_0$.  \item
Monotonicity: if $NF(M_0,M_1,M_2,M_3) \ and \ N_0=M_0, n<3
\rightarrow N_n\leq M_n\wedge N_0\leq N_n\leq N_3, (\exists
N^{*})[M_3\leq N^{*}\wedge N_3\leq N^{*}] \ then \ NF(N_0
\allowbreak ,N_1,N_2,N_3)$. \item Existence: For every
$N_0,N_1,N_2 \in K_\lambda \ if \ l\in \{1,2\}\rightarrow N_0 \leq
N_l \ and \ N_1\bigcap \allowbreak N_2=N_0 \ then \ there \ is N_3
\ s.t. \ NF(N_0,N_1,N_2,N_3).$ \item Uniqueness: Suppose for x=a,b
$NF(N_0,N_1,N_2,N^{x}_3)$. Then there is a joint embedding of
$N^a,N^b \ over \ N_1 \bigcup N_2$. \item  Symmetry:
$NF(N_0,N_1,N_2,N_3) \leftrightarrow NF(N_0,N_2,N_1,N_3)$. \item
Long transitivity: For $x=a,b$ let $\langle  M_{x,i}:i\leq
\alpha^* \rangle$ an increasing continuous sequence of models in
$K_\lambda$. Suppose $i<\alpha^* \rightarrow
NF(M_{a,i},M_{a,i+1},M_{b,i},\allowbreak M_{b,i+1})$. Then
$NF(M_{a,0},M_{a,\alpha^{*}},M_{b,0},M_{b,\alpha^{*}})$ \item $NF$
respects $\frak{s}$: if $NF(M_0,M_1,M_2,M_3)$ and $tp(a,M_0,M_1)
\in S^{bs}(M_0)$ then $tp(a,M_2,M_3)$ does not fork over $M_0$.
\end{enumerate}
\end{fact}

\begin{proposition}\label{amalgamation NF}
If $\langle M_{\alpha}:\alpha \leq \alpha^* \rangle$ is an
increasing continuous sequence of models and $M_0 \prec N$, then
there is an increasing continuous sequence $\langle
N_{\alpha}:\alpha \leq \alpha^* \rangle$ such that for
$\alpha<\alpha^*$
$NF(M_{\alpha},M_{\alpha+1},N_{\alpha},N_{\alpha+1})$ and $N,N_0$
are isomorphic over $M_0$.
\end{proposition}

\begin{proof}
We choose $(N^{temp}_\alpha,f_\alpha)$ by induction on $\alpha$
such that:
\begin{enumerate}
\item $\alpha \leq \theta \Rightarrow N^{temp}_\alpha \in
K_\lambda$. \item $(N_0^{temp},f_0)=(N,id_{M_0})$. \item The
sequence $\langle N^{temp}_\alpha:\alpha \leq \theta \rangle$ is
increasing and continuous. \item The sequence $\langle
f_\alpha:\alpha \leq \theta \rangle$ is increasing and continuous.
\item For $\alpha \leq \theta$, the function $f_\alpha$ is an
embedding of $M_\alpha$ to $N^{temp}_\alpha$. \item For
$\alpha<\theta$, we have
$NF(M_\alpha,M_{\alpha+1},N^{temp}_\alpha,N^{temp}_{\alpha+1})$.
\end{enumerate}

\begin{displaymath}
\xymatrix{ N=N_0^{temp} \ar[r]^{id} & N_2^{temp} \ar[rr]^{id} &&
N_\alpha^{temp} \ar[r]^{id}
& N_{\alpha+1}^{temp} \ar[rr]^{id} && N_\theta^{temp}\\
M_0 \ar[r]^{id} \ar[u]^{f_0} & M_2 \ar[rr]^{id} \ar[u]^{f_2} &&
M_\alpha \ar[r]^{id} \ar[u]^{f_\alpha} & M_{\alpha+1} \ar[rr]^{id}
\ar[u]^{f_{\alpha+1}} && M_\theta \ar[u]^{f_\theta}}
\end{displaymath}
\emph{Why is this possible?} For $\alpha=0$ see 2. For $\alpha$
limit define $N^{temp}_\alpha:=\bigcup
\{N_\beta^{temp}:\beta<\alpha \},\ f_\alpha:=\bigcup
\{f_\beta:\beta<\alpha \}$. By the induction hypothesis
$\beta<\alpha \Rightarrow f_\beta[M_\beta] \preceq N_\beta^{temp}$
and the sequences $\langle N_\beta^{temp}:\beta \leq \alpha
\rangle,\ \langle f_\beta:\beta \leq \alpha \rangle$ are
increasing and continuous. So by smoothness (Definition
\ref{1.1}.d) $f_\alpha[M_\alpha] \preceq N^{temp}_\alpha$.

For successor $\alpha$ we use the existence (clause c) in Fact
\ref{jrsh875.5.15}.

Now $f_\theta:M_\theta \to N_\theta^{temp}$ is an isomorphism.
Extend $f_\theta^{-1}$ to a function $g$ with domain
$N_\theta^{temp}$ and define $f:=g \restriction N$. By 2,3 $N
\preceq N_\theta^{temp}$. By 2, $f$ is an isomorphism over $M_0$.
Define $N_\alpha:=g[N^{temp}_\alpha]$. By 5, $f_\alpha[M_\alpha]
\preceq N^{temp}_\alpha$, so $M_\alpha \preceq N_\alpha$.
\end{proof}

\begin{proposition}\label{independence NF}
If $J$ is independent in $(M,M_0,M_1)$ and $NF(M_0,M_1,\allowbreak
M_2,M_3)$ then $J$ is independent in $(M,M_2,M_3)$.
\end{proposition}

\begin{proof}
By Proposition \ref{amalgamation NF}, the long transitivity in
Fact \ref{jrsh875.5.15} and the uniqueness in Fact
\ref{jrsh875.5.15}. We elaborate: By Definition \ref{definition of
independence} there is an independent sequence $\langle
N_{0,\alpha},a_\alpha:\alpha<\alpha^* \rangle ^\frown \langle
N_{0,\alpha^*} \rangle$ over $M$ such that
$J=\{a_\alpha:\alpha<\alpha^*\}$, $N_{0,0}=M_0$ and $M_1 \preceq
N_{0,\alpha^*}$. By Proposition \ref{amalgamation NF} there is an
increasing continuous sequence of models $\langle
N_{1,\alpha}:\alpha \leq \alpha^* \rangle$ such that for $\alpha
\leq \alpha^*$ we have
$NF(N_{0,\alpha},N_{0,\alpha+1},N_{1,\alpha},N_{1,\alpha+1})$ and
there is an isomorphism $f:M_2 \to N_{1,0}$ over $M_0$. Now we
prove:
\begin{enumerate}[(i)]
\item $NF(M_0,M_1,N_{1,0},N_{1,\alpha^*})$. \item The sequence
$\langle N_{1,\alpha},a_\alpha:\alpha<\alpha^* \rangle ^\frown
\langle N_{1,\alpha^*} \rangle$ is independent over $M$.
\end{enumerate}

Why is it enough? By assumption $NF(M_0,M_1,M_2,M_3)$, so by (i)
and the uniqueness in Fact \ref{jrsh875.5.15}
$(id_{M_1},id_{M_2},M_3)\allowbreak E_{M_0}(id \allowbreak
_{M_1},f,M_3^*)$. Therefore by (ii), the set $J$ is independent in
$(M,M_2,M_3)$.

\begin{proof}
\mbox{}
\begin{enumerate}[(i)]
\item By the long transitivity in Fact \ref{jrsh875.5.15}
 we have
$NF(M_0,N_{0,\alpha^*},N_{1,0},N_{1,\alpha^*})$. But $M_0 \preceq
M_1 \preceq N_{0,\alpha^*}$. So by the monotonicity in Fact
\ref{jrsh875.5.15} $NF(M_0,M_1\allowbreak
,N_{1,0},N_{1,\alpha^*})$. \item By Fact \ref{jrsh875.5.15} the
relation $NF$ respects the frame $\frak{s}$. But for
$\alpha<\alpha^*$
$NF(M_{0,\alpha},M_{0,\alpha+1},M_{1,\alpha},M_{1,\alpha+1})$
holds. So $tp(a_\alpha,M_{1,\alpha},\allowbreak M_{1,\alpha+1})$
does not fork over $M_{0,\alpha}$. Since the sequence $\langle
N_{0,\alpha}:\alpha<\alpha^* \rangle ^\frown \langle
N_{0,\alpha^*} \rangle$ is independent over $M$,
$tp(a_\alpha,M_{0,\alpha},M_{0,\alpha+1})$ does not fork over $M$.
So by the transitivity (Definition \ref{definition of a good frame
minus stability}), $tp(a_\alpha,M_{1,\alpha},M_{1,\alpha+1})$ does
not fork over $M$.
\end{enumerate}
\end{proof}
This ends the proof of Proposition \ref{independence NF}.
\end{proof}

\begin{proposition}\label{amalgamation NF + disjointness}
If $\langle M_{\alpha}:\alpha \leq \alpha^*+1 \rangle$ is an
increasing continuous sequence of models, then there is an
increasing continuous sequence $\langle N_{\alpha}:\alpha \leq
\alpha^* \rangle$ such that for $\alpha<\alpha^*$
$NF(M_{\alpha},M_{\alpha+1},N_{\alpha},N_{\alpha+1})$ and
$M_{\alpha^*+1} \preceq N_{\alpha^*}$.
\end{proposition}

\begin{proof}
By the proof of Proposition \ref{preparation exchange}.
\end{proof}

\begin{proposition}\label{finite continuity proposition}
The finitely independence has continuity. Equivalently: If
\begin{enumerate}
\item $\delta$ is a limit ordinal. \item $\langle M_\alpha:\alpha
\leq \delta+1 \rangle$ is increasing and continuous. \item $J
\subset M_{\delta+1}$ and it is finite. \item For $\alpha<\delta$
$J$ is independent in $(M,M_\alpha,M_{\delta+1})$.

\end{enumerate}
then $J$ is independent in $(M,M_\delta,M_{\delta+1})$.
\end{proposition}

\begin{proof}
By Proposition \ref{amalgamation NF + disjointness} there is
another increasing continuous sequence $\langle N_{\alpha}:\alpha
\leq \delta \rangle$ such that for $\alpha<\delta$
$NF(M_{\alpha},\allowbreak M_{\alpha+1},N_{\alpha},N_{\alpha+1})$
and $M_{\delta+1} \preceq N_{\delta}$.

Take $\alpha<\delta$ with $J \subset N_\alpha$. Since $J$ is
independent in $(M,M_\alpha,M_{\delta+1})$ and $J \subseteq
N_\alpha \preceq N_\delta \succeq M_{\delta+1}$, it follows that
$J$ is independent in $(M,M_\alpha,N_\alpha)$. By the long
transitivity in Fact \ref{jrsh875.5.15}, we have
$NF(M_\alpha,N_\alpha,M_\delta,N_\delta)$. Therefore by
Proposition \ref{independence NF} (where the models
$M,M_\alpha,N_\alpha,M_\delta,N_\delta$ which appear here stands
for the models $M,M_0,M_1,M_2,M_3$) $J$ is independent in
$(M,M_\delta,N_{\delta})$, hence in $(M,M_\delta,\allowbreak
M_{\delta+1})$.
\end{proof}

Finally, we prove $(*)$ of Proposition \ref{b implies d and c}.
Denote (*) for $\alpha^*$ by $(*)_{\alpha^*}$. We prove
$(*)_{\alpha^*}$ by induction on $\alpha^*$:

\case{Case a:} $\alpha^*=0$. The conclusion is actually assumption
(1).

\case{Case b:} $\alpha^*=\gamma+1$.
By the induction hypothesis $(*)_\gamma$ holds. We want to
conclude that $J-\{a_\alpha:\alpha<\gamma \}$ is finitely
independent in
$(M,M_\gamma,M^+)$. So we check the conditions:\\
(1): It is assumption (1).\\
(2): By assumption (2) $\{a_\alpha:\alpha<\alpha^*\} \subseteq J
\subseteq M^+$. But $\{a_\alpha:\alpha<\gamma\} \subseteq
\{a_\alpha:\alpha<\alpha^*\}$.\\
(3): By assumption (3) the sequence $\langle
M_\alpha,a_\alpha:\alpha<\alpha^* \rangle ^\frown \langle
M_{\alpha^*} \rangle$ is independent over $M$ by $K^{3,uq}$,
namely the sequence $\langle M_\alpha,a_\alpha:\alpha<\gamma
\rangle ^\frown \langle M_{\gamma},a_\gamma,M_{\alpha^*} \rangle$
is independent over $M$ by $K^{3,uq}$. So the sequence $\langle
M_\alpha,a_\alpha:\alpha<\gamma \rangle ^\frown \langle M_{\gamma}
\rangle$ is independent over $M$ by $K^{3,uq}$.\\
(4): By assumption (4) $M_{\alpha^*} \preceq M^+$. But by
assumption (3) $M_\gamma \preceq M_{\alpha^*}$. So $M_\gamma
\preceq M^{+}$.\\
Now by $(*)_\gamma$ the set $J-\{a_\alpha:\alpha<\gamma\}$ is
finitely independent in $(M,M_\gamma,M^+)$.

But by assumption 3, $(M_\gamma,M_{\alpha^*},a_{\alpha^*}) \in
K^{3,uq}$. So Proposition \ref{b_1} (substituting
$J-\{a_\alpha:\alpha<\gamma\},M,M_\gamma,M^+,a_{\alpha^*},M_{\alpha^*}$
instead of $J,M,M_0,M^+,a,\allowbreak M_1$ respectively), yields
that $J-\{a_\alpha:\alpha<\alpha^*\}$ is finitely independent in
$(M,M_{\alpha^*},M^+)$.

\case{Case c:} $\alpha^*$ is limit.  Let $J^*$ be a finite subset
of $J-\{a_\alpha:\alpha<\alpha^*\}$. We have to prove that $J^*$
is independent in $(M,M_{\alpha^*},M^+)$. For $\gamma<\alpha^*$
$J^* \subseteq J-\{a_\alpha:\alpha<\gamma\}$, so by $(*)_\gamma$,
$J^*$ is independent in $(M,M_\gamma,M^+)$. Therefore by
Proposition \ref{finite continuity proposition}, $J^*$ is
independent in $(M,M_{\alpha^*},M^+)$.

This ends the proof of Theorem \ref{the main theorem}.
\end{proof}

\subsection*{ Acknowledgment} We thank Boaz Tsaban for his useful suggestions and
comments.

{10}

\end{document}